\documentclass[a4paper,12pt]{article}
\usepackage{a4wide}
\usepackage{amsmath}
\usepackage{amssymb}
\usepackage{amsthm}
\usepackage{latexsym}
\usepackage{graphicx}
\usepackage[english]{babel}
\usepackage{makeidx}
\usepackage{mathtools}

\newtheorem{prop}[subsection]{Proposition}

\newtheorem{teor}[subsection]{Theorem}

\newtheorem{cor} [subsection]{Corollary}

\DeclarePairedDelimiter\floor{\lfloor}{\rfloor}

\newcommand{\pa}{p_{\mathbf a}}
\newcommand{\fa}{f_{\mathbf a}}
\newcommand{\Fa}{F_{\mathbf a}}

\newcommand{\da}{p_{\mathbf a}}

\newcommand{\Pa}{P_{\mathbf a}}

\newcommand{\stir}{\genfrac{[}{]}{0pt}{}}
\newcommand{\stirr}{\genfrac{\{}{\}}{0pt}{}}

\def\lcm{\operatorname{lcm}}
\def\gcd{\operatorname{gcd}}

\numberwithin{equation}{section}

\begin{document}
\selectlanguage{english}
\frenchspacing

\large
\begin{center}
\textbf{On the restricted partition function II}

Mircea Cimpoea\c s and Florin Nicolae
\end{center}
\normalsize

\begin{abstract}
Let $\mathbf a=(a_1,\ldots,a_r)$ be a vector of positive integers. In continuation of a previous paper 
we present other formulas for the restricted partition function $\pa(n): = $ 
the number of integer solutions $(x_1,\dots,x_r)$ to
$\sum_{j=1}^r a_jx_j=n$ with $x_1\geq 0, \ldots, x_r\geq 0$.


\end{abstract}

\section*{Introduction}

\footnotetext[1]{The first author was supported 
by a grant of the Romanian National Authority for Scientific
Research, CNCS - UEFISCDI, project number PN-II-ID-PCE-2011-1023.}

\indent Let $\mathbf a:=(a_1,a_2,\ldots,a_r)$ be a sequence of positive integers, $r\geq 1$. 
The \emph{restricted partition function} associated to $\mathbf a$ is $\pa:\mathbb N \rightarrow \mathbb N$, 
$\pa(n):=$ the number of integer solutions $(x_1,\ldots,x_r)$ of $\sum_{i=1}^r a_ix_i=n$ with $x_i\geq 0$. 
Let $D$ be a common multiple of $a_1,\ldots,a_r$. 

Sylvester \cite{sylvester},\cite{sylv} decomposed the restricted partition in a sum of ``waves'': 
$$\pa(n)=\sum_{j\geq 1} W_{j}(n,\mathbf a),$$
where the sum is taken over all distinct divisors $j$ of the components of $\mathbf a$ and showed that for each such $j$, 
$W_j(n,\mathbf a)$ is the coefficient of $t^{-1}$ in
$$ \sum_{0 \leq \nu <j,\; \gcd(\nu,j)=1 } \frac{\rho_j^{-\nu n} e^{nt}}{(1-\rho_j^{\nu a_1}e^{-a_1t})\cdots (1-\rho_j^{\nu a_r}e^{-a_rt}) },$$
where $\rho_j=e^{\frac{2\pi i}{j}}$ and $\gcd(0,0)=1$ by convention.

Note that $W_{j}(n,\mathbf a)$'s are quasi-polynomials of period $j$.
( A \emph{quasi-polynomial} of period $j$ is a numerical function $f(n)$ such that there exists $j$ polynomials $P_1(n),P_2(n),\ldots,P_j(n)$ such that $f(n)=P_i(n)$ if $n\equiv i (\bmod j)$.) The first wave $\Pa(n):=W_1(\mathbf a,n)$ is called the \emph{polynomial part} of $\pa(n)$. 

Glaisher \cite{glaisher} made computations of the Sylvester waves in particular cases.
Fel and Rubinstein \cite{rubfel} proved formulas for the Sylvester waves using Bernoulli and Euler polynomials of higher order.
Rubinstein \cite{rub} showed that all Sylvester waves can be expressed in terms of Bernoulli 
polynomials only. Bayad and Beck \cite[Theorem 3.1]{babeck} proved an explicit expression 
of the partition function $\pa(n)$ in terms of Bernoulli-Barnes polynomials and the Fourier Dedekind sums, in the case that $a_1,\ldots,a_r$
are pairwise coprime. 
Beck, Gessler and Komatsu \cite[page 2]{beck}, Dilcher and Vignat \cite[Theorem 1.1]{dilcher} proved explicit formulas for the polynomial part of $\pa(n)$. 

As a continuation of \cite{lucrare} we present here other formulas for $\pa(n)$ and for the Sylvester waves. 
In Proposition $2.2$ we prove that
$$ \pa(n) = \sum_{j=0}^{\floor*{\frac{n}{D}}} \binom{r+j-1}{j} \fa(n-jD), $$
where $ \fa(n)=\#\{(j_1,\ldots,j_r):\; a_1j_1+\cdots+a_rj_r=n,\; 0\leq j_k\leq \frac{D}{a_k}-1,\;1\leq k\leq r\}$. This result is
similar to Theorem $1$ of Rodseth and Seller \cite{rodseth}.

In Corollary $2.4$ (compare \cite[Theorem 2]{rodseth}) we prove the congruence
$$  (r-1)!\pa(n) \equiv 0 \; \bmod \; (j+k+1)(j+k+2)\cdot \ldots \cdot (j+r-1), $$
where $k= \floor*{\frac{n}{D}} - \left\lceil \frac{n+a_1+\cdots+a_r}{D} \right\rceil + r$.

In Proposition $4.2$ we prove that
$$ W_{j}(n,\mathbf a) = \frac{1}{D(r-1)!} \sum_{m=1}^r \sum_{\ell=1}^{j} \rho_j^{\ell} \sum_{k=m-1}^{r-1} 
\stir{r}{k} (-1)^{k-m+1} \binom{k}{m-1} \cdot$$ $$\cdot \sum_{\substack{0\leq j_1\leq \frac{D}{a_1}-1,\ldots, 0\leq j_r\leq \frac{D}{a_r}-1 \\ a_1j_1+\cdots+a_rj_r \equiv \ell (\bmod j)}} D^{-k} (a_1j_1+\cdots+a_rj_r)^{k-m+1} n^{m-1},$$
where $\stir{r}{k}$ are unsigned Stirling numbers of the first kind.

The \emph{Bernoulli numbers} are defined by the identity $$\frac{t}{e^t-1}=\sum_{\ell=0}^{\infty}\frac{t^{\ell}}{\ell !}B_{\ell}.$$
If $\gcd(a_i,a_j)=1$ for all $i\neq j$ we prove in Proposition $4.3$ that
$$\pa(n) = \sum_{m=1}^r \frac{(-1)^{r-m}}{(a_1\cdots a_r)(m-1)!}\sum_{i_{1}+\cdots+i_r=r-m}  \frac{B_{i_1}\cdots B_{i_r}}{i_1!\cdots i_r!} a_1^{i_1} \cdots a_r^{i_r} n^{m-1} +  \frac{1}{D(r-1)!} \cdot$$
$$ \cdot \sum_{j \neq 1} \sum_{\ell=1}^{j} \rho_j^{\ell}  \sum_{k=0}^{r-1} \frac{1}{D^k}\stir{r}{k} (-1)^{k} 
\sum_{\substack{0\leq j_1\leq \frac{D}{a_1}-1,\ldots, 0\leq j_r\leq \frac{D}{a_r}-1 \\ a_1j_1+\cdots+a_rj_r \equiv \ell (\bmod j)}} D^{-k} (a_1j_1+\cdots+a_rj_r)^{k}, $$
where $j|a_i$ for some $1\leq i\leq r$. Another formulas for $\pa(n)$ in the case that $a_1,\ldots,a_r$ are pairwise coprimes were proved in \cite[Theorem C, pag 113]{comtet}, \cite[Theorem 3.1]{babeck} and \cite{komatsu}.

Let $$\sum_{n=0}^{\infty}\pa(n)z^n = \frac{1}{(1-z^{a_1}) \ldots (1-z^{a_r})} =
\sum_{\lambda^D=1}\sum_{\ell=1}^{m(\lambda)}\frac{c_{\lambda,\ell}}{(\lambda-z)^{\ell}},$$
where $m(\lambda)$ is the multiplicity of $\lambda$ as a root of $(1-z^{a_1}) \ldots (1-z^{a_r})$.

In Proposition $4.4$ we prove that 
$$ c_{\rho_j,m}=\frac{\rho_j^m (m-1)!}{D} \sum_{t=m}^{m(\rho_j)} (-1)^{t-m}\stirr{t}{m} \sum_{\ell=1}^{j} \rho_j^{\ell} \sum_{k=m-1}^{r-1} 
\stir{r}{k} (-1)^{k-m+1} \binom{k}{m-1} \cdot $$ $$ \cdot \sum_{\substack{0\leq j_1\leq \frac{D}{a_1}-1,\ldots, 0\leq j_r\leq \frac{D}{a_r}-1 \\ a_1j_1+\cdots+a_rj_r \equiv n (\bmod j)}} D^{-k} (a_1j_1+\cdots+a_rj_r)^{k-m+1},$$
where 
$\stirr{t}{m}$ are Stirling numbers of the second kind.

In Proposition $4.5$ we prove that
$$c_{1,m} = \frac{(m-1)!}{a_1\cdots a_r} 
\sum_{\ell=m}^r (-1)^{\ell-m}\frac{\stirr{\ell}{m}} {(\ell-1)!} \sum_{i_1+\cdots+i_r=r-\ell}
\frac{B_{i_1}\cdots B_{i_r}}{i_1!\cdots i_r!}a_1^{i_1} \cdots a_r^{i_r}.$$
In the case $\mathbf a=(1,2,\ldots,r)$ we reprove O'Sullivan's formulas \cite{cormac} for Rademacher's coefficients $c_{01m}$, see Corollary $4.6$.

Given a sequence of positive integers $\mathbf a = (a_1,\ldots,a_r)$ with $gcd(a_1,\ldots,a_r)=1$, the \emph{Frobenius number} of $\mathbf a$,
denoted by $F(\mathbf a)=F(a_1,\ldots,a_r)$ is the largest integer $n$ with the property that $\pa(n)=0$. If $\gcd(a_i,a_j)=1$ for all $i\neq j$, we prove in Proposition $5.2$ that $F(A_1,\ldots,A_r)=D(r-1)-A_1-\cdots-A_r$, where $A_1:=\frac{D}{a_1},\ldots,A_r:=\frac{D}{a_r}$. This is a particular case of \cite[Theorem 2.7]{racz} and appears also in \cite[Theorem 1(a)]{tripathi}.

\section{Preliminaries}

Let $\mathbf a:=(a_1,a_2,\ldots,a_r)$ be a sequence of positive integers, $r\geq 1$. The \emph{restricted partition function} 
associated to $\mathbf a$ is $\da:\mathbb N \rightarrow \mathbb N$, 
$\pa(n):=$ the number of integer solutions $(x_1,\ldots,x_r)$ of $\sum_{i=1}^r a_ix_i=n$ with $x_i\geq 0$. 

Let $D$ be a common multiple of $a_1,a_2,\ldots,a_r$. 
Bell \cite{bell} has proved that $\pa(n)$ is a quasi-polynomial of degree $r-1$, with the period $D$, i.e. 
$$\pa(n)=d_{\mathbf a,r-1}(n)n^{r-1}+\cdots+d_{\mathbf a,1}(n)n +d_{\mathbf a,0}(n),$$
where $d_{\mathbf a,m}(n+D)=d_{\mathbf a,m}(n)$ for $0\leq m\leq r-1$ and $n\geq 0$, and $d_{\mathbf a,r-1}(n)$ is not identically zero.

In the following, we recall several results from our previous paper \cite{lucrare}.

\begin{teor}(\cite[Theorem 2.8(1)]{lucrare})
For $0\leq m\leq r-1$ and $n\geq 0$ we have 
$$d_{\mathbf a, m}(n) =\frac{1}{(r-1)!}  \sum_{\substack{0\leq j_1\leq \frac{D}{a_1}-1,\ldots, 0\leq j_r\leq \frac{D}{a_r}-1 \\ a_1j_1+\cdots+a_rj_r \equiv n (\bmod D)}} 
\sum_{k=m}^{r-1} \stir{r}{k} (-1)^{k-m} \binom{k}{m} D^{-k} (a_1j_1 + \cdots + a_rj_r)^{k-m}.$$
\end{teor}

\begin{cor}(\cite[Corollary 2.10]{lucrare}) We have
$$ \pa(n) = \frac{1}{(r-1)!} \sum_{\substack{0\leq j_1\leq \frac{D}{a_1}-1,\ldots, 0\leq j_r\leq \frac{D}{a_r}-1 \\ 
a_1j_1+\cdots+a_rj_r \equiv n (\bmod D)}} \prod_{\ell=1}^{r-1} (\frac{n-a_1j_{1}- \cdots -a_rj_r}{D}+\ell ).$$
\end{cor}

\begin{cor}(\cite[Corollary 2.12]{lucrare})
For $n\geq 0$ we have $\pa(n)=0$ if and only if $n<a_1j_1+\cdots+a_rj_r$ for all 
$0\leq j_1\leq \frac{D}{a_1}-1,\ldots, 0\leq j_r\leq \frac{D}{a_r}-1$ with
$a_1j_1+\cdots+a_rj_r \equiv n (\bmod D)$.
\end{cor}

We also recall the following result of Beck, Gessler and Komatsu \cite[page 2]{beck}. See also \cite[Corollary 2.11]{lucrare}.

\begin{teor}
The polynomial part of $\pa(n)$ is
$$P_{\mathbf a}(n) := \frac{1}{a_1\cdots a_r}\sum_{u=0}^{r-1}\frac{(-1)^u}{(r-1-u)!}\sum_{i_1+\cdots+i_r=u} 
\frac{B_{i_1}\cdots B_{i_r}}{i_1!\cdots i_r!}a_1^{i_1}\cdots a_r^{i_r} n^{r-1-u}.$$
\end{teor}

\section{A formula and a congruence for $\pa(n)$}

Let $\mathbf a:=(a_1,a_2,\ldots,a_r)$ be a sequence of positive integers, $r\geq 1$. It holds that
$$\sum_{n=0}^{\infty}\pa(n)z^n = \frac{1}{(1-z^{a_1})\cdots(1-z^{a_r})},\; |z|<1.$$
Let $D$ be a common multiple of $a_1,\ldots,a_r$. Let 
$$F_{\mathbf a}(z):= \frac{(1-z^D)^r}{(1-z^{a_1})\cdots(1-z^{a_r})}=\prod_{i=1}^r (1+z^{a_i}+\cdots+z^{a_i(\frac{D}{a_i}-1)}). $$
Let $d:=rD-a_1-\cdots-a_r$. Since $F_{\mathbf a}(z) = F_{\mathbf a}(\frac{1}{z})$, it follows that
$$F_{\mathbf a}(z)=:f_{\mathbf a}(d)z^d + \cdots +f_{\mathbf a}(1)z + f_{\mathbf a}(0)$$ is a reciprocal polynomial, 
that is $\fa(d-n)=\fa(n)$ for $0\leq n\leq d$. Let $\fa(n):=0$ for $n\geq d+1$. It holds that
$$ \fa(n)=\#\{(j_1,\ldots,j_r):\; a_1j_1+\cdots+a_rj_r=n,\; 0\leq j_k\leq \frac{D}{a_k}-1,\;1\leq k\leq r\}. $$

From the power series identity
$$\sum_{n=0}^{\infty}\fa(n)z^n = \Fa(z)= (1-z^D)^r\sum_{n=0}^{\infty}\pa(n)z^n = \sum_{n=0}^{\infty}\pa(n) \sum_{j=0}^r \binom{r}{j}(-1)^j z^{n+jD}$$
it follows that
\begin{equation} \fa(n) = \sum_{j=0}^{\floor*{\frac{n}{D}} } \binom{r}{j}(-1)^j \pa(n-jD),\; n\geq 0. \end{equation}

\begin{prop}
For $n\geq 0$ we have that 
$$\fa(n) = \frac{1}{(r-1)!} \sum_{j=0}^{\floor*{\frac{n}{D}} } \binom{r}{j}(-1)^j \sum_{\substack{0\leq j_1\leq 
\frac{D}{a_1}-1,\ldots, 0\leq j_r\leq \frac{D}{a_r}-1 \\ a_1j_1+\cdots+a_rj_r \equiv n (\bmod D)}} 
\prod_{\ell=1}^{r-1} (\frac{n-a_1j_{1}- \cdots -a_rj_r}{D}+\ell-j)$$
\end{prop}

\begin{proof}
It follows from Corollary $1.2$ and $(2.1)$.
\end{proof}

\begin{prop}(compare \cite[Theorem 1]{rodseth})
It holds that
 $$ \pa(n) = \sum_{j=0}^{\floor*{\frac{n}{D}}} \binom{r+j-1}{j} \fa(n-jD),\;n\geq 0. $$
\end{prop}

\begin{proof}
Denote $k:=\floor*{\frac{n}{D}}$.
From $(2.1)$ we get the following system of linear equations in the indeterminates $\pa(n-jD)$, $0\leq j\leq k$
\begin{equation} \sum_{j=t}^{k} \binom{r}{j-t}(-1)^{j-t} \pa(n-jD) = \fa(n-tD), \; 0\leq t\leq k. \end{equation}
It follows that  $$ \pa(n) =\sum_{j=0}^k (-1)^j \Delta_{j}\fa(n-jD), $$
where $\Delta_0=1$ and $ \Delta_{j} =- \sum_{i=0}^{j-1} \binom{r}{j-i} \Delta_i$.
Using induction on $j\geq 0$ it follows that $\Delta_j=(-1)^{j}\binom{r+j-1}{j}$ for all $0\leq j\leq k$. 
Hence $$ \pa(n) = \sum_{j=0}^k \binom{r+j-1}{j} \fa(n-jD). $$
\end{proof}

\begin{cor}
For $n\geq 0$ it holds that
$$ (r-1)!\pa(n) = \sum_{j= \left\lceil \frac{n+\sigma}{D} \right\rceil - r}^{\floor*{\frac{n}{D}}} (j+1)\cdots (j+r-1) f_{\mathbf a}(n-jD), $$
where $\sigma=a_1+\cdots +a_r$.
\end{cor}

\begin{proof}
For $n-jD > rD-\sigma$ it holds that $f_{\mathbf a}(n-jD)= 0$. From Proposition $2.2$ it follows that
$$(r-1)!\pa(n) = \sum_{j=0}^{\floor*{\frac{n}{D}}} (r-1)! \binom{r+j-1}{j} \fa(n-jD) = $$
$$ = \sum_{j= \left\lceil \frac{n+\sigma}{D} \right\rceil - r}^{\floor*{\frac{n}{D}}} (j+1)\cdots (j+r-1) f_{\mathbf a}(n-jD). $$
\end{proof}


\begin{cor}(compare \cite[Theorem 2]{rodseth})
For $n\geq 0$ it holds that
$$  (r-1)!\pa(n) \equiv 0 \; \bmod \; (j+k+1)(j+k+2)\cdot \ldots \cdot (j+r-1), $$
where $k= \floor*{\frac{n}{D}} - \left\lceil \frac{n+\sigma}{D} \right\rceil + r$, $\sigma=a_1+\ldots +a_r$.
\end{cor}

\begin{proof}
For $\left\lceil \frac{n+\sigma}{D} \right\rceil - r \leq j \leq  \floor*{\frac{n}{D}}$ it holds that 
$$ (j+1)\cdots (j+r-1) \equiv 0 \; \bmod (j+k+1) \cdots (j+r-1). $$
Apply now Corollary $2.3$.
\end{proof}

\section{Quasi-polynomials}

Let $p:\mathbb N \rightarrow \mathbb C$ be a quasi-polynomial of degree $r-1 \geq 0$,
$$ p(n):=d_{r-1}(n)n^{r-1}+\cdots+d_1(n)n +d_0(n), $$
where $d_m(n)$'s are periodic functions with integral period $D>0$ and $d_{r-1}(n)$ is not identically zero.

According to \cite[Proposition 4.4.1]{stanley}, we have
$$\sum_{n=0}^{\infty}p(n)z^n = \frac{L(z)}{M(z)},$$ where $L(z),M(z)\in \mathbb C[z]$, every zero $\lambda$ of $M(z)$ satisfies $\lambda^D=1$
(provided $\frac{L(z)}{M(z)}$ has been reduced to lowest terms), and $\deg L(z)<\deg M(z)$. Moreover,
$$p(n)=\sum_{\lambda^D=1} P_{\lambda}(n) \lambda^{-n},$$ where each $P_{\lambda}(n)$ is a polynomial function with $\deg P_{\lambda}(n) \leq m(\lambda)-1$, where $m(\lambda)$ is the multiplicity of $\lambda$ as a root of $M(z)$.
We define the \emph{polynomial part} of $p(n)$ to be the polynomial function $P(n):=P_1(n)$.

Let $\gamma\in \mathbb C$ with $\gamma^D=1$. It holds that 
$$p_{\gamma}(n):=\gamma^{n}p(n) = \sum_{\lambda^D=1} P_{\lambda}(n) (\gamma\cdot \lambda^{-1})^n,$$
hence $P_{\gamma}(n)$ is the polynomial part of $p_{\gamma}(n)$.

\begin{prop} (\cite[Proposition 2.5]{lucrare})
It holds that $$P_{\gamma}(n) = R_{\gamma,m(\gamma)}n^{r-1} + \cdots + R_{\gamma,2}n + R_{\gamma,1},$$
where $R_{\gamma,m}=\frac{1}{D}\sum_{v=0}^{D-1} \gamma^{v}d_{m-1}(v)$, $1\leq m\leq m(\gamma)$.
\end{prop}

Consider the decomposition
\begin{equation}
  \sum_{n=0}^{\infty}p(n)z^n = \frac{L(z)}{M(z)} = \sum_{M(\lambda)=0}\sum_{\ell=1}^{m(\lambda)}\frac{c_{\lambda,\ell}}{(\lambda-z)^{\ell}}. 
\end{equation}
Let $\gamma$ be a root of $M(z)$. Since the decomposition $(3.1)$ is unique, it follows that 
$$ \sum_{n=0}^{\infty}P_{\gamma}(n)\gamma^{-n}z^n = \sum_{\ell=1}^{m(\gamma)}\frac{c_{\gamma,\ell}}{(\gamma-z)^{\ell}} = 
\sum_{\ell=1}^{m(\gamma)} \frac{c_{\gamma,\ell}}{\gamma^{\ell}}\left(\sum_{n=0}^{\infty} \gamma^{-n}z^n \right)^{\ell} = \sum_{\ell=1}^{m(\gamma)} 
\frac{c_{\gamma,\ell}}{\gamma^{\ell}} \sum_{n=0}^{\infty} \binom{n+\ell-1}{\ell-1} \gamma^{-n}z^n. $$
It follows that
\begin{equation}
P_{\gamma}(n) =  \sum_{\ell=1}^{m(\gamma)}  \frac{c_{\gamma,\ell}}{\gamma^{\ell}}\binom{n+\ell-1}{\ell-1}.
\end{equation}

The \emph{Stirling numbers of the second kind}, denoted by $\stirr{n}{k}$, count the number of ways to partition a set of $n$ labelled 
objects into $k$ nonempty unlabelled subsets. They are related with the unsigned Stirling numbers of the first kind by
\begin{equation}
\sum_{k=0}^n \stirr{n}{k} \sum_{\ell=0}^k (-1)^{\ell} \stir{k}{\ell} = (-1)^n
\end{equation}

\begin{prop}
For each $1\leq m \leq m(\gamma)$ it holds that
$$ c_{\gamma,m} = \gamma^{m}(m-1)! \sum_{\ell=m}^{m(\gamma)}(-1)^{\ell-m}\stirr{\ell}{m}\frac{1}{D}\sum_{v=0}^{D-1} \gamma^{v}d_{\ell-1}(v).\Box$$
\end{prop}

\begin{proof}
From Proposition $3.1$ and $(3.2)$ it follows that $R_{\gamma,m}=0$ for all $m > m(\gamma)$ and 
\begin{equation}
R_{\gamma,m}=\sum_{\ell=1}^{m(\gamma)}\frac{c_{\gamma,\ell}}{\gamma^{\ell}(\ell-1)!} \stir{\ell}{m-1},\;1\leq m\leq m(\gamma).
\end{equation}
From $(3.3)$ and $(3.4)$ it follows that
\begin{equation}
 c_{\gamma,m} = \gamma^{m}(m-1)! \sum_{\ell=m}^{m(\gamma)}(-1)^{\ell-m}\stirr{\ell}{m} R_{\gamma,\ell},
\end{equation}
hence, the conclusion follows from Proposition $3.1$.
\end{proof}

\section{Sylvester waves and the partial fraction decomposition of $\sum_{n=0}^{\infty}\pa(n)z^n$}

Let $\mathbf a = a_1,\ldots,a_r$ a sequence of positive integers. We write $\pa(n)$ as a sum of waves
$$ \pa(n)=\sum_{j} W_{j}(n,\mathbf a),$$ 
where the sum is taken over the $j\geq 1$ with $j|a_i$ for some $1\leq i\leq r$. We have that 
\begin{equation} W_j(n,\mathbf a) =  P_{\mathbf a, \rho_j}(n) \rho_j^{-n},\end{equation}
where $\rho_j:=e^{\frac{2\pi i}{j}}$ and $P_{\mathbf a,\rho_j}(n)$ is the polynomial part of the 
quasi-polynomial $\rho_j^n\pa(n)$.

\begin{prop}
We have that
$$ W_{j}(n,\mathbf a) =\rho_j^{-n}(R_{j,m(j)}\cdot n^{m(j)-1} + \cdots + R_{j,2}\cdot n + R_{j,1}), $$
where 
$m(j) = \#\{i:\; j|a_i\}$ and $R_{j,m}=\frac{1}{D}\sum_{v=0}^{D-1}\rho_j^v d_{\mathbf a,m-1}(v)$ for $1\leq m\leq m(j)$.
\end{prop}

\begin{proof}
It follows from Proposition $3.1$ and $(4.1)$.
\end{proof}

\begin{prop}
For any positive integer $j$ with $j|a_i$ for some $1\leq i\leq r$, we have that:
$$ W_{j}(n,\mathbf a) = \frac{1}{D(r-1)!} \sum_{m=1}^r \sum_{\ell=1}^{j} \rho_j^{\ell} \sum_{k=m-1}^{r-1} 
\stir{r}{k} (-1)^{k-m+1} \binom{k}{m-1} \cdot$$ $$\cdot \sum_{\substack{0\leq j_1\leq \frac{D}{a_1}-1,\ldots, 0\leq j_r\leq \frac{D}{a_r}-1 \\ a_1j_1+\cdots+a_rj_r \equiv \ell (\bmod j)}} D^{-k} (a_1j_1+\cdots+a_rj_r)^{k-m+1} n^{m-1}.$$
\end{prop}

\begin{proof}
It follows from Proposition $4.1$ and Theorem $1.1$.
\end{proof}

\begin{prop}
If $a_1,\ldots,a_r$ are pairwise coprimes then
$$\pa(n) = \sum_{m=1}^r \frac{(-1)^{r-m}}{(a_1\cdots a_r)(m-1)!}\sum_{i_{1}+\cdots+i_r=r-m}  \frac{B_{i_1}\cdots B_{i_r}}{i_1!\cdots i_r!} a_1^{i_1} \cdots a_r^{i_r} n^{m-1} +  \frac{1}{D(r-1)!} \cdot$$
$$ \cdot \sum_{j \neq 1} \sum_{\ell=1}^{j} \rho_j^{\ell}  \sum_{k=0}^{r-1} \frac{1}{D^k}\stir{r}{k} (-1)^{k} 
\sum_{\substack{0\leq j_1\leq \frac{D}{a_1}-1,\ldots, 0\leq j_r\leq \frac{D}{a_r}-1 \\ a_1j_1+\cdots+a_rj_r \equiv \ell (\bmod j)}} D^{-k} (a_1j_1+\cdots+a_rj_r)^{k}, $$
where $j|a_i$ for some $1\leq i\leq r$.
\end{prop}

\begin{proof}
Since $a_1,\ldots,a_r$ are pairwise coprimes, it follows that $W_j(\mathbf a, n)$ is a quasi-polynomial of degree $0$. Hence, the conclusion
follows from Proposition $4.1$ and Proposition $4.2$.
\end{proof}

Another formulas for $\pa(n)$ in the case that $a_1,\ldots,a_r$ are pairwise coprimes were proved in \cite[Theorem C, pag 113]{comtet}, \cite[Theorem 3.1]{babeck} and \cite{komatsu}.

We consider the decomposition
$$ \sum_{n=0}^{\infty}\pa(n)z^n = \frac{1}{(1-z^{a_1})\cdots (1-z^{a_r})} = \sum_{\lambda} \sum_{\ell=1}^{m(\lambda)} 
\frac{c_{\lambda,\ell}}{(\lambda-z)^{\ell}}, $$
where the sum it taken over the $\lambda$'s with $\lambda^{a_i}=1$ for some $1\leq i\leq r$.

\begin{prop}
Let $j\geq 1$. For $1\leq m\leq m(j)$ we have that
$$ c_{\rho_j,m}=\frac{\rho_j^m (m-1)!}{D} \sum_{t=m}^{m(j)} (-1)^{t-m}\stirr{t}{m} \sum_{\ell=1}^{j} \rho_j^{\ell} \sum_{k=m-1}^{r-1} 
\stir{r}{k} (-1)^{k-m+1} \binom{k}{m-1} \cdot $$ $$ \cdot \sum_{\substack{0\leq j_1\leq \frac{D}{a_1}-1,\ldots, 0\leq j_r\leq \frac{D}{a_r}-1 \\ a_1j_1+\cdots+a_rj_r \equiv n (\bmod j)}} D^{-k} (a_1j_1+\cdots+a_rj_r)^{k-m+1}.$$
\end{prop}

\begin{proof}
It follows from Proposition $3.2$ and Theorem $4.1$.
\end{proof}

\begin{prop}
For $1\leq m\leq r$ it holds that
$$c_m = \frac{(m-1)!}{a_1\cdots a_r} 
\sum_{\ell=m}^r (-1)^{\ell-m}\frac{\stirr{\ell}{m}} {(\ell-1)!} \sum_{i_1+\cdots+i_r=r-\ell}\frac{B_{i_1}\cdots B_{i_r}}{i_1!\cdots i_r!}a_1^{i_1} \cdots a_r^{i_r},$$
where $c_m:=c_{1,m}$.
\end{prop}

\begin{proof}
It follows from Theorem $1.4$, Proposition $3.1$ and Proposition $4.4$.
\end{proof}

Let $r\geq 1$, $\mathbf r:=(1,2,\ldots,r)$. \emph{Rademacher's coefficients} $c_{hk\ell}(r)$ are defined by
$$\sum_{n=0}^{\infty}p_{\mathbf r}(n)z^n = \frac{1}{(1-z)(1-z^2)\cdots(1-z^r)} = \sum_{0\leq h<k\leq r,\;(h,k)=1}\sum_{\ell=1}^{ \left\lfloor \frac{r}{k} \right\rfloor} \frac{c_{hk\ell}(r)}{(z-\omega_{hk})^{\ell}}, $$
where $\omega_{hk}:=e^{2\pi i \frac{h}{k}}$. In the previous notations $c_{hk\ell}(r)=(-1)^{\ell}c_{\omega_{hk},\ell}$. 
As a direct consequence of Proposition $4.5$ we get the following result of C.\ O'Sullivan \cite{cormac}:

\begin{cor}(\cite[Proposition 2.3]{cormac} )
For $1\leq m\leq r$ it holds that
$$c_{01m}(r) =   \frac{(-1)^{r}(m-1)!}{r!} 
\sum_{\ell=m}^r \frac{\stirr{\ell}{m}}{(\ell-1)!} \sum_{i_1+\cdots+i_r=r-\ell}\frac{B_{i_1}\cdots B_{i_r}}{i_1!\cdots i_r!}1^{i_1}2^{i_2} \cdots r^{i_r}.\Box$$ 
\end{cor}


\section{Frobenius number}

Given a sequence of positive integers $\mathbf a = (a_1,\ldots,a_r)$ with $gcd(a_1,\ldots,a_r)=1$, the \emph{Frobenius number} of $\mathbf a$,
denoted by $F(\mathbf a)=F(a_1,\ldots,a_r)$ is the largest integer $n$ with the property that $\pa(n)=0$.

\begin{prop}
Let  $\mathbf a = (a_1,\ldots,a_r)$ with $gcd(a_1,\ldots,a_r)=1$ and $D=\lcm(a_1,\ldots,a_r)$. We have that 
$$F(a_1,\ldots,a_r)\leq D(r-1)-a_1-\cdots-a_r.$$
\end{prop}

\begin{proof}
Let $n$ be an integer with $\pa(n)=0$. Since the map $$\varphi: \mathbb Z/a_1\mathbb Z \times \cdots \times \mathbb Z/a_r\mathbb Z \rightarrow \mathbb Z/D\mathbb Z,\;\varphi(\hat j_1,\ldots,\hat j_r):=\overline{a_1j_1+\cdots+a_rj_r}$$
is a surjective morphism, it follows that there exists some integers $0\leq j_1 \leq \frac{D}{a_1}-1, \ldots,0\leq j_1 \leq \frac{D}{a_1}-1$  such that $a_1j_1+\cdots+a_rj_r \equiv n (\bmod D)$. 

From Corollary $1.3$ it follows that $n<a_1j_1+\cdot+a_rj_r$, hence
$$n\leq a_1j_1+\cdots+a_rj_r - D \leq (D-1)r-a_1-\cdots-a_r. $$
\end{proof}

The following corollary is a particular case of \cite[Theorem 2.7]{racz} and appears 
also in \cite[Theorem 1(a)]{tripathi}.

\begin{cor}
Let $\mathbf a = (a_1,\ldots,a_r)$ such that $\gcd(a_i,a_j)=1$ for all $i\neq j$,
 $D=\lcm(a_1,\ldots,a_r)=a_1\cdots a_r$, $A_i:=\frac{D}{a_i}$, $1\leq i\leq r$. It holds that
$$ F(A_1,\ldots,A_r)=D(r-1)-A_1-\cdots-A_r.$$
\end{cor}

\begin{proof}
It holds that $D=\lcm(A_1,\ldots,A_r)$. From Lemma $5.1$ it follows that $$F(A_1,\ldots,A_r)\leq D(r-1)-A_1-\cdots-A_r.$$

Suppose that $D(r-1)-A_1-\cdots-A_r = A_1j_1+\cdots+A_rj_r$ with $j_k\geq 0$ for $1\leq k\leq r$, hence
$$D(r-1)=A_1(j_1+1)+\cdots+A_r(j_r+1),$$
$$ r-1=\frac{j_1+1}{a_1}+\cdots+\frac{j_r+1}{a_r}. $$
Since $\gcd(a_i,a_j)=1$ it follows that $a_k|(j_k+1)$ for all $1\leq k\leq r$.
Since $j_k\geq 0$ we get $$r-1 = \frac{j_1+1}{a_1}+\cdots+\frac{j_r+1}{a_r}\geq r,$$ a contradiction. So $F(A_1,\ldots,A_r)\geq D(r-1)-A_1-\cdots-A_r.$
\end{proof}

{}

\vspace{2mm} \noindent {\footnotesize
\begin{minipage}[b]{15cm}
Mircea Cimpoea\c s, Simion Stoilow Institute of Mathematics, Research unit 5, P.O.Box 1-764,\\
Bucharest 014700, Romania, E-mail: mircea.cimpoeas@imar.ro
\end{minipage}}

\vspace{2mm} \noindent {\footnotesize
\begin{minipage}[b]{15cm}
Florin Nicolae, Simion Stoilow Institute of Mathematics, P.O.Box 1-764,\\
Bucharest 014700, Romania, E-mail: florin.nicolae@imar.ro
\end{minipage}}

\end{document}